\documentclass[a4paper,11pt]{amsart}
\usepackage[matrix,arrow,tips,curve]{xy}
\usepackage[english]{babel}
\usepackage{amsmath}
\usepackage{amssymb}
\usepackage{tikz}
\usepackage{mathrsfs}
\usepackage{enumerate}
\usepackage{graphicx}
\usepackage{hyperref}
\usetikzlibrary{trees}
\usetikzlibrary{arrows}

\newtheorem*{thm*}{Theorem}

\theoremstyle{definition}

\newtheorem{teo}{Theorem}[section]
\newtheorem{lema}[teo]{Lemma}

\newtheorem{cor}[teo]{Corollary}

\newtheorem{rem}{Remark}[section]

\definecolor{wwwwww}{rgb}{0.4,0.4,0.4}

\hypersetup{pdfpagemode=UseNone}
\hypersetup{pdfstartview=FitH}
		
\begin{document}
\thanks{This study was financed in part by the Coordena\c{c}\~ao de Aperfei\c{c}oamento de Pessoal de N\'ivel Superior - Brasil (CAPES) Finance Code 001}
\title[\resizebox{6.1in}{!}{Rigidity theorems for constant weighted mean curvature hypersurfaces}]{Rigidity theorems for constant weighted mean curvature hypersurfaces}

\author[Saul Ancari]{Saul Ancari}
\address{\sc Saul Ancari\\
Instituto de Matem\'atica e Estat\'istica, Universidade Federal Fluminense, Campus Gragoat\'a, Rua Alexandre Moura 8 - S\~ao Domingos\\
24210-200 Niter\'oi, Rio de Janeiro\\ Brazil}
\email{sa\_ancari@id.uff.br}

\author[Igor Miranda]{Igor Miranda}
\address{\sc Igor Miranda\\
Instituto de Matem\'atica e Estat\'istica, Universidade Federal Fluminense, Campus Gragoat\'a, Rua Alexandre Moura 8 - S\~ao Domingos\\
24210-200 Niter\'oi, Rio de Janeiro\\ Brazil}
\email{igor\_miranda@id.uff.br}


\begin{abstract}
In this article, we study hypersurfaces $\Sigma\subset \mathbb{R}^{n+1}$ with constant weighted mean curvature. Recently, Wei-Peng proved a rigidity theorem for CWMC hypersurfaces that generalizes Le-Sesum classification theorem for self-shrinker. More specifically, they showed that a complete CWMC hypersurface with polynomial volume growth, bounded norm of the second fundamental form and that satisfies $|A|^2H(H-\lambda)\leq H^2/2$ must either be a hyperplane or a generalized cylinder. We generalize this result by removing the bound condition on the norm of the second fundamental form. Moreover, we prove that under some conditions if the reverse inequality holds then the hypersurface must either be a hyperplane or a generalized cylinder. As an application of one of the results proved in this paper, we will obtain another version of the classification theorem obtained by the authors of this article, that is, we show that under some conditions, a complete CWMC hypersurface with $H\geq 0$ must either be a hyperplane or a generalized cylinder.
\end{abstract}

\maketitle

\section{Introduction}
One of the main interests in the study of the mean curvature flow theory is to understand possible singularities that the flow goes through. The singularity models for these flows can be associated to hypersurfaces that satisfy the following mean curvature condition
\[
H=\frac{\langle x , \nu \rangle}{2}
\]
where $H$, $x$ and $\nu$ stand for the mean curvature of $\Sigma$, the position vector in $\mathbb{R}^{n}$ and the unit normal vector of $\Sigma$, respectively. Such hypersurfaces are known as self-shrinkers. Another characterization of the self-shrinkers is that they are critical points of the weighted area functional
\stepcounter{thm}
\begin{eqnarray}\label{11}
F(\Sigma)=\int_\Sigma e^{-\frac{|x|^2}{4}}dv.
\end{eqnarray}

There is a great interest in studying two-sided smooth hypersurfaces $\Sigma\subset\mathbb{R}^{n+1}$ which are critical points of the functional (\ref{11}) for variations $G:(-\varepsilon,\varepsilon)\times\Sigma \rightarrow \mathbb{R}^{n+1}$ that preserve enclosed weighted volume. These variations can be represented by functions $u:\Sigma \rightarrow \mathbb{R}$ defined by 
\[
u(x)=\langle \partial_t G(0,x),\nu(x)\rangle
\]
such that $\int_\Sigma u\  e^{-|x|^2/4}dv =0$, where $\nu$ is the normal vector of $\Sigma$. It is well known that these hypersurfaces satisfy the following condition
\begin{eqnarray*}
H=\frac{\langle x, \nu\rangle}{2}+\lambda,
\end{eqnarray*}
where $\lambda \in \mathbb{R}$. Such hypersurfaces are known as $\lambda$-hypersurfaces or as constant weighted mean curvature hypersurfaces. Throughout this paper, whenever $\Sigma$ satisfies the mean curvature condition above, they will be called CWMC hypersurfaces. Self-shrinkers, hyperplanes, spheres and cylinders are some examples of CWMC hypersurfaces. The study of such hypersurfaces arises in geometry and probability as solutions to the Gaussian isoperimetric problem. These hypersurfaces were first studied by Cheng-Wei \cite{cheng2018complete} and McGonagle-Ross \cite{mcgonagle2015hyperplane}. Since then, there has been much interest about classification results for CWMC hypersurfaces, for instance \cite{sun2018compactness},\cite{guang2018gap},\cite{cheng2016rigidity}, etc. Le and Sesum \cite{le2011blow} proved a gap theorem showing that if a complete embedded self-shrinker $\Sigma\subset\mathbb{R}^{n+1}$ with polynomial volume growth ( there exist a constant $C>0$ and $r_0>0$ such that $V(B_r(0)\cap\Sigma)\leq Cr^\alpha$, for all $r\geq r_0$ and for some $\alpha>0$ ) and such that the norm of the second fundamental form satisfies $|A|^2< 1/2$, then $\Sigma$ should be a hyperplane. Cao and Li \cite{cao2013gap} extended this theorem to a more general result for arbitrary codimension proving that a complete embedded self-shrinker $\Sigma^n\subset \mathbb{R}^{n+p}$ with polynomial volume growth and $|A|^2\leq 1/2$ has to be a generalized cylinder. Later, Guang \cite{guang2018gap}, Cheng, Ogata and Wei \cite{cheng2016rigidity} proved rigidity theorems for CWMC that generalized Le-Sesum's theorem. Recently, Wei and Peng proved another generalization for CWMC hypersurfaces. More specifically, they showed the following result.
\begin{teo}\cite{wei2019note}
Let $\Sigma\subset \mathbb{R}^{n+1}$ be a complete embedded CWMC hypersurface with polynomial volume growth. If  the norm of the second fundamental form is bounded and\\
\[
|A|^2H(H-\lambda)\leq \frac{H^2}{2}
\]
then $\Sigma$ is either a hyperplane or a generalized cylinder $S^k_r(0)\times \mathbb{R}^{n-k}$, $1\leq k \leq n$. 
\end{teo}
In this paper, we will prove a similar result without the bound assumption on the norm of the second fundamental form. In fact, we prove the following theorem.
\begin{teo}\label{teointro}
Let $\Sigma\subset \mathbb{R}^{n+1}$ be a complete embedded CWMC hypersurface. If  the following properties hold:\\
(i) $|A|^2H(H-\lambda)\leq \frac{H^2}{2}$,\\
(ii) $\frac{1}{k^2}\int_{B^\Sigma_{2k}(p)\setminus B^\Sigma_k(p)} H^2e^{-\frac{|x|^2}{4}}\rightarrow 0$, when $k\rightarrow \infty$, for a fixed point $p\in \Sigma$,\\
then $\Sigma$ is either a hyperplane or a generalized cylinder $S^k_r(0)\times \mathbb{R}^{n-k}$, $1\leq k \leq n$.
\end{teo}
\rem: Notice that if $\Sigma$ has polynomial volume growth, then the condition (ii) is satisfied.
\rem: The polynomial volume growth condition appears in most of the classification theorems and it is needed since there are self-shrinkers that do not satisfy this condition, for instance see \cite{halldorsson2012self}.\\

The authors of this paper proved the following classification theorem.
\begin{teo}\cite{ancari2019volume}
Let $\Sigma\subset \mathbb{R}^{n+1}$ be a complete embedded CWMC hypersurface. If $\Sigma$ satisfies the following properties:\\
(i) $H-\lambda\geq 0$,\\
(ii) $\lambda\left( \text{tr}A^3(H-\lambda)+\frac{|A|^2}{2}\right)\leq 0$,\\
(iii) $\frac{1}{k^2}\int_{B^\Sigma_{2k}(p)\setminus B^\Sigma_k(p)} |A|^2e^{-\frac{|x|^2}{4}}\rightarrow 0$, when $k\rightarrow \infty$, for a fixed point $p\in \Sigma$,\\
then $\Sigma$ must be either a hyperplane or $S^{k}_r(0)\times \mathbb{R}^{n-k}$, $1\leq k \leq n$.
\end{teo}
An application of Theorem \ref{teointro} is another version of the theorem above.
\begin{teo}\label{AM19}
Let $\Sigma\subset \mathbb{R}^{n+1}$ be a complete embedded CWMC hypersurface. If $\Sigma$ satisfies the following properties:\\
(i) $H\geq 0$,\\
(ii) $\lambda\left( \text{tr}A^3(H-\lambda)+\frac{|A|^2}{2}\right)\leq 0$,\\
(iii) $\frac{1}{k^2}\int_{B^\Sigma_{2k}(p)\setminus B^\Sigma_k(p)} |A|^2e^{-\frac{|x|^2}{4}}\rightarrow 0$, when $k\rightarrow \infty$, for a fixed point $p\in \Sigma$,\\
then $\Sigma$ must be either a hyperplane or $S^{k}_r(0)\times \mathbb{R}^{n-k}$, $1\leq k \leq n$.
\end{teo}
Finally, our third main result is the following.
\begin{teo}\label{teointro2}Let $\Sigma\subset \mathbb{R}^{n+1}$ be a complete embedded CWMC hypersurface. If the following conditions are satisfied:\\
(i) $|A|^2(H-\lambda)\geq \frac{H}{2}$,\\
(ii) $\frac{1}{k^2}\int_{B^\Sigma_{2k}(p)\setminus B^\Sigma_k(p)} e^{-\frac{|x|^2}{4}}\rightarrow 0$, when $k\rightarrow \infty$, for a fixed point $p\in \Sigma$,\\
then $\Sigma$ either is a hyperplane or a generalized cylinder $S^k_r(0)\times \mathbb{R}^{n-k}$, $1\leq k \leq n$.
\end{teo}
Note that for the case $H\geq 0$, the condition (i) from Theorem \ref{teointro2} is the reverse inequality from the condition (i) in Theorem \ref{teointro}.\\

This work is divided into three sections. In section \ref{s2}, we recall some notations, basic tools and key formulas for CWMC hypersurfaces. In section \ref{s3}, we prove the main results of this paper.
\section{Preliminaries}\label{s2}
In this section, we will establish some notations and recall some definitions and basic results.\\

 Let $\Sigma$ be a hypersurface embedded on $\mathbb{R}^{n+1}$, endowed with the metric g induced by the Euclidean metric. We will denote by $\nabla$, $\Delta$ and $d\sigma$, the connection, Laplacian  and volume form, respectively. The second fundamental form of $(\Sigma,g)$ at $p\in \Sigma$ is defined as 
\[
A(X,Y)= \langle\overline{\nabla}_XY,\nu\rangle,
\]
where $X,Y\in T_p\Sigma$ and $\nu$ is normal unit vector. We will denote by $h_{ij}= A(e_i,e_j)$, where $\lbrace e_i\rbrace$ is an orthonormal basis of $T_p\Sigma$. The mean curvature of $\Sigma$ is defined as $H=-\sum_i h_{ii}$ . Throughout this work we will denote $f=|x|^2/4$ and the drifted Laplacian by
\[
\mathcal{L}=\Delta -\frac{1}{2}\langle x,\nabla\rangle.
\]
This operator is self-adjoint over $L^2_f(\Sigma)$. More specifically, we have the following lemma which is a consequence of Stokes' theorem.
\begin{lema}\label{pl1} Let $\Sigma\subset\mathbb{R}^{n+1}$ be a smooth hypersurface, $u\in C^1_0(\Sigma)$ and $v\in C^2(\Sigma)$. Then 
\[
\int_\Sigma u\mathcal{L}ve^{-f}=-\int_\Sigma \langle \nabla u, \nabla v\rangle e^{-f}.
\]
\end{lema}
In the next section, we will need the following results:
\begin{lema}\label{l1}If $\Sigma\subset\mathbb{R}^{n+1}$ is a CWMC hypersurface then
\stepcounter{thm}
\begin{eqnarray}
\mathcal{L}(H-\lambda)+(H-\lambda)|A|^2= \frac{H}{2},\label{1}\\
\stepcounter{thm}
\mathcal{L}|A|+\left(|A|^2-\frac{1}{2}\right)|A|= \frac{|\nabla A|^2-|\nabla |A||^2}{|A|}-\frac{\lambda \text{tr}A^3}{|A|}.\label{2}
\end{eqnarray}
\end{lema}
\begin{proof}
For a proof, see \cite{guang2018gap}.
\end{proof}
\section{Rigidity theorems}\label{s3}
In this section, we will prove the main theorems of this paper, see some consequences and applications.\\

\noindent\textit{Proof of Theorem \ref{teointro}}
Let $\varphi\in C_0^\infty(\Sigma)$, from (\ref{1}) and using hypothesis (i) we have
\begin{eqnarray*}
\int_\Sigma \varphi^2|\nabla H|^2e^{-f}&=& -\int_\Sigma \varphi^2H\mathcal{L}He^{-f}-\int_\Sigma \langle \nabla\varphi^2, H\nabla H\rangle e^{-f}\\
									   &=& -\int_\Sigma \varphi^2\left(\frac{H^2}{2}-|A|^2H(H-\lambda)\right)e^{-f}-2\int_\Sigma \langle H\nabla\varphi, \varphi\nabla H\rangle e^{-f}\\
									   &\leq & \int_\Sigma \left(2H^2|\nabla \varphi|^2+\frac{1}{2}\varphi^2|\nabla H|^2\right)e^{-f}.					
\end{eqnarray*}
Therefore,
\[
\frac{1}{2}\int_\Sigma \varphi^2|\nabla H|^2e^{-f}\leq 2\int_\Sigma H^2|\nabla \varphi|^2e^{-f}.
\]
Choosing a sequence $\varphi_k\in C_0^\infty$ such that $\varphi_k(x)=1$ for $x\in B_k^\Sigma(p)$, $\varphi_k(x)=0$ for $x\in\Sigma\setminus B_{2k}^\Sigma(p)$ and $|\nabla \varphi_k|\leq 1/k$, by the monotone convergence theorem and hypothesis (ii), we get
\[
|\nabla H|=0, 
\]
which implies that $H$ is constant. If $H=0$ from equation (\ref{1}), we conclude that $\Sigma$ must be a hyperplane. Otherwise, $|A|^2=\frac{H}{2(H-\lambda)}$ and $|A|$ is also constant. From equation (\ref{2}) it follows that
\stepcounter{thm}
\begin{eqnarray}\label{16}
|\nabla A|^2=\left(|A|^2-\frac{1}{2}\right)|A|^2+\lambda \text{tr}A^3.
\end{eqnarray}
On the other hand, since the Simon's equation holds, that is
\[
\frac{1}{2}\Delta |A|^2= |\nabla A|^2-\langle A, \text{Hess}\ H\rangle -H\text{tr}A^3-|A|^4,
\]
it follows that
\stepcounter{thm}
\begin{eqnarray}\label{13}
|\nabla A|^2= H\text{tr}A^3+|A|^4.
\end{eqnarray}
Therefore, combining (\ref{16}) and (\ref{13}), we obtain
\[
\text{tr}A^3= \frac{|A|^2}{2(\lambda -H)}.
\]
From (\ref{13}) and the equation above, we have
\[
|\nabla A|^2=\frac{H|A|^2}{2(\lambda -H)}+|A|^4
\]
and since $|A|^2=\frac{H}{2(H-\lambda)}$, we conclude that
\[
|\nabla A|^2=0.
\]
By Lawson's theorem, $\Sigma$ must be a generalized cylinder.
\begin{flushright}
$\Box$
\end{flushright}
\rem\label{remthac}: It is possible to prove a more general result changing the hypothesis (ii) by\\
$\frac{1}{k^2}\int_{B^\Sigma_{2k}(p)\setminus B^\Sigma_k(p)} H^qe^{-f}\rightarrow 0$, when $k\rightarrow \infty$, for some even number $q\geq 2$ and a fixed point $p\in \Sigma$. The idea of the proof is to use a similar argument for $H^{q-1}\mathcal{L}H$.\\

A consequence of Theorem \ref{teointro} is the following.
\begin{cor}\label{corfinal2}Let $\Sigma\subset \mathbb{R}^{n+1}$ be a complete embedded CWMC hypersurface. If $\Sigma$ satisfies the following conditions:\\
(i)$|A|^2H(H-\lambda)\leq \frac{H^2}{2}$,\\
(ii) $\frac{1}{k^2}\int_{B^\Sigma_{2k}(p)\setminus B^\Sigma_k(p)} e^{-f}\rightarrow 0$, when $k\rightarrow \infty$, for a fixed point $p\in \Sigma$,\\
then $\Sigma$ is either a hyperplane or a generalized cylinder $S^k_r(0)\times \mathbb{R}^{n-k}$, $1\leq k \leq n$.
\end{cor}
\begin{proof}
First, we will see that the condition (i) implies that $\sup_{x\in\Sigma}H<+\infty$. In fact, if\newline $\sup_{x\in\Sigma}H=+\infty$, then there exists a sequence $\lbrace p_j\rbrace_{j\in\mathbb{N}}$ in $\Sigma$ such that $H(p_j)\rightarrow +\infty$ when $j\rightarrow +\infty$. Therefore, for $j$ large enough $H(p_j)>0$ and $\frac{\lambda}{H(p_j)}<1$. From hypothesis (i), we have
\begin{eqnarray*}
|A|^2(p_j)\leq \frac{1}{2}\left(1-\frac{\lambda}{H(p_j)}\right)^{-1}.
\end{eqnarray*}
Since $\left(1-\frac{\lambda}{H(p_j)}\right)^{-1}\rightarrow 1$ when $j\rightarrow +\infty$, then the left side of the inequality above is bounded which implies that $|A|^2(p_j)$ is bounded. Therefore, $H^2(p_j)$ is also bounded, but this contradicts the assumption that $H(p_j)\rightarrow +\infty$.\\
For a fixed point $p\in \Sigma$, we have
\[
\frac{1}{k^2}\int_{B^\Sigma_{2k}(p)\setminus B^\Sigma_k(p)} H^2e^{-f}\leq \frac{\sup_{x\in\Sigma}H^2}{k^2}\int_{B^\Sigma_{2k}(p)\setminus B^\Sigma_k(p)} e^{-f}
\]
and when $k\rightarrow \infty$, by the hypothesis (ii) the left side of the inequality goes to zero and by Theorem \ref{teointro}, the result follows.
\end{proof}
Recently, Cheng-Wei \cite{cheng2018complete} proved that a CWMC hypersurface with constant mean curvature must be locally isometric to a generalized cylinder. Using the same argument that was used at the end of the proof of Theorem \ref{teointro}, it is possible to conclude the following corollary.
\begin{cor}\label{cor1}If $\Sigma\subset \mathbb{R}^{n+1}$ is a complete embedded CWMC hypersurface with constant mean curvature, then $\Sigma$ is either a hyperplane or a generalized cylinder $S^k_r(0)\times \mathbb{R}^{n-k}$, $1\leq k \leq n$.
\end{cor}
\begin{proof}
If $H$ is constant, from equation (\ref{2}) we have
\[
\frac{H^2}{2}=|A|^2H(H-\lambda).
\]
The rest of the proof follows as in the proof of Theorem \ref{teointro}.
\end{proof}
\rem: Using the argument from the proof of Wei-Peng's theorem \cite{wei2019note}, one can also obtain the same result.\\

Note that in the proof of Theorem \ref{teointro}, one can obtain a condition which implies that $\Sigma$ must be a hyperplane. More specifically, the following holds.
\begin{cor}\label{l2}
Let $\Sigma\subset \mathbb{R}^{n+1}$ be a complete embedded CWMC hypersurface. If  the following properties hold:\\
(i) $0\leq H\leq \lambda$,\\
(ii) $\frac{1}{k^2}\int_{B^\Sigma_{2k}(p)\setminus B^\Sigma_k(p)} H^2e^{-f}\rightarrow 0$, when $k\rightarrow \infty$, for a fixed point $p\in \Sigma$,\\
then $\Sigma$ is a hyperplane.
\end{cor}
\begin{proof}
Since $0\leq H\leq \lambda$, the following condition holds
\[
|A|^2H(H-\lambda)\leq 0\leq \frac{H^2}{2}.
\]
Moreover, from the proof of Theorem \ref{teointro} and hypothesis (i), we have
\begin{eqnarray*}
0=|A|^2H(H-\lambda)- \frac{H^2}{2}\leq -\frac{H^2}{2},
\end{eqnarray*}
which implies that $H=0$ and $\Sigma$ must be a hyperplane.
\end{proof}
Using the corollary above, it is possible to prove the second main theorem of this paper. Most of the proof of this result is in \cite{ancari2019volume}, but we will include the entire proof in order to keep this work self-contained.

\noindent\textit{Proof of Theorem \ref{AM19}}
For $\lambda\leq 0$, from Lemma \ref{l1} we have
\stepcounter{thm}
\begin{eqnarray}\label{15}
\mathcal{L}(H-\lambda)+(|A|^2-\frac{1}{2})(H-\lambda)=\frac{\lambda}{2}\leq 0.
\end{eqnarray}
Since $H-\lambda\geq 0$, by the maximum principle we can conclude that either $H-\lambda=0$ or $H-\lambda>0$. If $H-\lambda=0$, then from (\ref{15}) we conclude that $\lambda=0$, which implies that $\Sigma$ is a self-shrinker. Moreover, Colding-Minicozzi proved in \cite{colding2012generic} that a self-shrinker such that $H=0$ has to be a hyperplane. If $\lambda>0$ and $H-\lambda=0$ at some point $p\in \Sigma$, from hypothesis (ii)
\[
0\geq \lambda\left( \text{tr}A^3(H-\lambda)+\frac{|A|^2}{2}\right)=\frac{\lambda|A|^2}{2}
\]
at $p \in \Sigma$. This implies that $|A|(p)=0$, but this contradicts the fact that $H(p)>0$. Therefore, for $\lambda>0$, either $H-\lambda>0$ or $\lambda-H>0$. If $\lambda-H>0$, since $H^2\leq n|A|^2$, we have
\[
\frac{1}{k^2}\int_{B^\Sigma_{2k}(p)\setminus B^\Sigma_k(p)}H^2e^{-f}\leq \frac{n}{k^2}\int_{B^\Sigma_{2k}(p)\setminus B^\Sigma_k(p)}|A|^2.
\]
Therefore if $0\leq H< \lambda$, by Corollary \ref{l2} $\Sigma$ must be a hyperplane. To conclude the proof, we only need to study the case $H-\lambda>0$.\\
 
If $H-\lambda>0$, we will prove that either $|A|=0$ or $|A|=C(H-\lambda)$, for $C>0$. Consider the functions $u=H-\lambda$ and $v= \sqrt{|A|^2+\varepsilon}$. Computing $\mathcal{L}(u)$ and $\mathcal{L}(v)$, we get
\stepcounter{thm}
\begin{eqnarray}\label{7}
\mathcal{L}u+\left(|A|^2-\frac{1}{2}\right)u= \frac{\lambda}{2}
\end{eqnarray}
and
\[
\mathcal{L}v+\left(|A|^2-\frac{1}{2}\right)v= \frac{|\nabla A|^2-|\nabla v|^2}{v}+\left(|A|^2-\frac{1}{2}\right)\frac{\varepsilon}{v}-\frac{\lambda \text{tr}A^3}{v}.
\]
Since
\stepcounter{thm}
\begin{eqnarray}\label{8}
\frac{|\nabla A|^2-|\nabla v|^2}{v}\geq 0,\ \ \mathcal{L}v+\left(|A|^2-\frac{1}{2}\right)v\geq -\frac{\varepsilon}{2v}-\frac{\lambda \text{tr}A^3}{v}.
\end{eqnarray}
Let us consider $w=\frac{v}{u}$. Thus, from (\ref{7}) we get
\begin{eqnarray*}
\mathcal{L}v&=&w\mathcal{L}u +2\langle \nabla w, \nabla u\rangle +u\mathcal{L}w\\
			&=&w\left(\frac{\lambda}{2}-|A|^2+\frac{1}{2}\right)+2\langle \nabla w, \nabla u \rangle+u\mathcal{L}w.
\end{eqnarray*}
By (\ref{8}), we have
\[
u\mathcal{L}w\geq -\frac{1}{v}\left(\frac{\varepsilon}{2}+\lambda \text{tr}A^3\right)-\frac{\lambda w}{2}-2\langle \nabla w, \nabla u\rangle.
\]
Using hypothesis (ii), it is possible to conclude that
\stepcounter{thm}
\begin{eqnarray}\label{9}
\frac{u}{v}\left(-\frac{\varepsilon}{2}-\lambda \text{tr}A^3 \right)-\frac{v\lambda}{2}\geq -\frac{\varepsilon H}{2v}.
\end{eqnarray}
From the inequality above, we obtain
\stepcounter{thm}
\begin{eqnarray}\label{10}
\mathcal{L}w\geq -\frac{\varepsilon H}{2vu^2}-2\langle \nabla w, \nabla \log u\rangle.
\end{eqnarray}
For a function $\varphi\in C_0^\infty(\Sigma)$, using integration by parts and (\ref{10}), we get
\begin{eqnarray*}
\int_\Sigma \varphi^2|\nabla w|^2e^{-f}&=&-\int_\Sigma \varphi^2 w \mathcal{L}w e^{-f}-\int_\Sigma 2\varphi w\langle \nabla \varphi, \nabla w\rangle e^{-f}\\
									   &\leq & 2\int_\Sigma \varphi^2w\langle \nabla w, \nabla \log u\rangle e^{-f}+\frac{\varepsilon}{2}\int_\Sigma \frac{\varphi^2wH}{vu^2}e^{-f}-\int_\Sigma2\varphi w\langle \nabla \varphi,\nabla w\rangle e^{-f}\\
									   &=& 2\int_\Sigma \langle \varphi \nabla w, \varphi w \nabla \log u-w\nabla \varphi\rangle e^{-f}+\frac{\varepsilon}{2}\int_\Sigma \frac{\varphi^2H}{u^3}e^{-f}\\
									   &\leq & \frac{1}{2}\int_\Sigma \varphi^2 |\nabla w|^2 e^{-f}+2\int_{\Sigma} w^2|\varphi\nabla \log u-\nabla \varphi|^2 e^{-f}+\frac{\varepsilon}{2}\int_\Sigma \frac{\varphi^2H}{u^3}e^{-f}.
\end{eqnarray*}
Therefore,
\begin{eqnarray*}
\int_\Sigma \varphi^2|\nabla w|^2e^{-f}\leq 4 \int_\Sigma w^2|\varphi\nabla \log u-\nabla \varphi|^2e^{-f}+\varepsilon\int_\Sigma \frac{\varphi^2H}{u^3}e^{-f}.
\end{eqnarray*}
Choosing $\varphi=\psi u$, $\psi\in C_0^\infty(\Sigma)$, we have
\begin{eqnarray*}
\int_\Sigma \psi^2 u^2|\nabla w|^2 e^{-f}&\leq & 4 \int_\Sigma v^2 |\nabla \psi|^2e^{-f}+\varepsilon\int_\Sigma \psi^2 e^{-f}+\varepsilon\lambda\int_\Sigma \frac{\psi^2}{u}e^{-f}.
\end{eqnarray*}
For $\lambda\geq 0$, choosing $\varepsilon = 0$ we obtain
\begin{eqnarray*}
\int_\Sigma \psi^2 u^2|\nabla w|^2 e^{-f}&\leq & 4 \int_\Sigma v^2 |\nabla \psi|^2e^{-f}.
\end{eqnarray*}
Consider a sequence $\psi_k\in C_0^\infty(\Sigma)$, such that $\psi_k=1$ in $B_k^\Sigma(p)$, $\psi_k=0$ in $\Sigma\setminus B_{2k}^\Sigma(p)$ and $|\nabla \psi_k|\leq 1/k$ for every $k$, we have
\begin{eqnarray*}
\int_\Sigma \psi_k^2 u^2|\nabla w|^2 e^{-f}&\leq & 4 \int_{B_{2k}^\Sigma(p)\setminus B_k^\Sigma(p)} v^2 |\nabla \psi_k|^2e^{-f}\\
										   &\leq & \frac{4}{k^2} \int_{B_{2k}^\Sigma(p)\setminus B_k^\Sigma(p)} v^2 e^{-f}\\
										   &= & \frac{4}{k^2} \int_{B_{2k}^\Sigma(p)\setminus B_k^\Sigma(p)} |A|^2 e^{-f}.
\end{eqnarray*}
By the monotone convergence theorem and hypothesis (iii), we get
\[
\int_\Sigma u^2\left|\nabla \left(\frac{|A|}{H-\lambda}\right)\right|^2e^{-f}=0,
\]
which implies that $|A|=C(H-\lambda)$, for a constant $C>0$. 

For $\lambda<0$, we have 
\[
\int_\Sigma \psi^2 u^2|\nabla w|^2 e^{-f}\leq 4 \int_\Sigma v^2 |\nabla \psi|^2e^{-f}+\varepsilon\int_\Sigma \psi^2 e^{-f}.
\]
As in the other case, consider a sequence $\psi_k\in C_0^\infty(\Sigma)$, such that $\psi_k=1$ in $B_k^\Sigma(p)$, $\psi_k=0$ in $\Sigma\setminus B_{2k}^\Sigma(p)$ and $|\nabla \psi_k|\leq 1/k$ for every $k$, hence we get
\begin{eqnarray*}
\int_\Sigma \psi_k^2 u^2|\nabla w|^2 e^{-f}&\leq & 4 \int_\Sigma v^2 |\nabla \psi_k|^2e^{-f}+\varepsilon\int_\Sigma \psi_k^2 e^{-f}\\
										   &\leq & \frac{4}{k^2}\int_{B_{2k}^\Sigma(p)\setminus B_k^\Sigma(p)} |A|^2 e^{-f} +\frac{4\varepsilon}{k^2}\int_{B_{2k}^\Sigma(p)\setminus B_k^\Sigma(p)} e^{-f}+\varepsilon\int_{B_{2k}^\Sigma(p)} e^{-f}.
\end{eqnarray*}
Choosing $\varepsilon = \left(k\int_{B_{2k}^\Sigma(p)} e^{-f}\right)^{-1}$, we have
\begin{eqnarray*}
\int_\Sigma \psi_k^2 u^2|\nabla w|^2 e^{-f}\leq \frac{4}{k^2}\int_{B_{2k}^\Sigma(p)\setminus B_k^\Sigma(p)} |A|^2 e^{-f} +\frac{4}{k^3}+\frac{1}{k}.
\end{eqnarray*}
Hence, by hypothesis (iii) we obtain
\begin{eqnarray*}
\lim_{k\rightarrow \infty} \int_\Sigma \psi_k^2 u^2|\nabla w|^2 e^{-f}=0.
\end{eqnarray*}
If the set
\[
\mathcal{A}=\lbrace p\in \Sigma; |A|(p)=0\rbrace
\]
is not empty, consider $\mathcal{B}=\Sigma\setminus \mathcal{A}$. Since $\mathcal{B}$ is an open set, let $p\in \mathcal{B}$ and $B^\Sigma_r(p)\subset \mathcal{B}$. For k sufficiently large, $B^\Sigma_r(p)\subset \text{supp}\psi_k$ and $\psi_k=1$ in $B^\Sigma_r(p)$. Hence
\[
\lim_{k\rightarrow \infty} \int_{B^\Sigma_r(p)}  u^2|\nabla w|^2 e^{-f}=0.
\]
By the dominated convergence theorem, we conclude that $|A|/(H-\lambda)$ is constant in $B^\Sigma_r(p)$. Since p is arbitrary, it is possible to conclude that $|A|/(H-\lambda)$ is constant in $\mathcal{B}$. Since $\mathcal{A}\neq \emptyset$, using a continuity argument, we conclude that $|A|=0$. If $\mathcal{A}= \emptyset$, by the dominated convergence theorem
\[
\int_\Sigma u^2\left|\nabla \left(\frac{|A|}{H-\lambda}\right)\right|^2e^{-f}=0,
\]
which implies $|A|=C(H-\lambda)$ for a constant $C>0$. Hence, when $H-\lambda>0$, we conclude that either $|A|=0$ or $|A|=C(H-\lambda)$ for a constant $C>0$. If $|A|=0$, then $\Sigma$ is a hyperplane. Otherwise, since $|A|=C(H-\lambda)$
\begin{eqnarray*}
\mathcal{L}|A|&=&\frac{|A|}{H-\lambda}\mathcal{L}(H-\lambda)\nonumber\\
			  &=& \frac{|A|\lambda}{2(H-\lambda)} +\left(\frac{1}{2}-|A|^2\right)|A|.
\end{eqnarray*}
On the other hand
\begin{eqnarray*}
\mathcal{L}|A| = \left(\frac{1}{2}-|A|^2\right)|A| +\frac{|\nabla A|^2 - |\nabla |A||^2}{|A|}-\frac{\lambda \text{tr}A^3}{|A|}.
\end{eqnarray*}
Hence, from the equations above we have
\begin{eqnarray*}
\frac{|\nabla A|^2 - |\nabla |A||^2}{|A|}&=& \frac{|A|\lambda}{2(H-\lambda)}+\frac{\lambda \text{tr}A^3}{|A|}\\
										 &=& \frac{\lambda}{(H-\lambda)|A|}\left(\text{tr}A^3(H-\lambda)+\frac{|A|^2}{2}\right).
\end{eqnarray*}
Using the hypothesis (ii) and the equality above, we conclude that 
\stepcounter{thm}
\begin{eqnarray}\label{12}
|\nabla |A||=|\nabla A|.
\end{eqnarray}
Fixing $p\in\Sigma$ and $\lbrace E_i\rbrace_{1\leq i \leq n}$ a orthonormal basis for $T_p\Sigma$, (\ref{12}) implies that for each $k$ there exists a constant $C_k$ such that
\[
h_{ijk}= C_kh_{ij}
\]
for all $i,j$. Considering a base such that $h_{ij}=\lambda_i\delta_{ij}$, by the Codazzi equation, we have
\[
h_{ijk}=0
\]
unless $i=j=k$. If $\lambda_i\neq 0$ and $i\neq j$ then
\[
0=h_{iij}=C_j\lambda_i.
\]
It follows that $C_j=0$. Hence, if the rank of the matrix $(h_{ij})$ is at least two at p, then $\nabla A(p)=0$. To show that $\nabla A=0$, let us fix $q\in\Sigma$ and suppose that $\lambda_1(q)$ and $\lambda_2(q)$ are the largest eigenvalues of $(h_{ij})(q)$. Define the following set
\[
\Lambda=\{q\in\Sigma;\lambda_1(q)=\lambda_1(p),\lambda_2(q)=\lambda_2(p)\}.
\]
Using the continuity of the $\lambda_i's$, it is possible to prove that the set $\Lambda$ is open and closed. Since $p\in \Lambda$ and $\Sigma$ is connected, $\Lambda=\Sigma$. Therefore, $\nabla A=0$ everywhere on $\Sigma$. Hence, $\Sigma$ is a isoparametric hypersurface and by a theorem proved by Lawson in \cite{lawson1969local}, $\Sigma$ must be $S^k_r(0)\times \mathbb{R}^{n-k}$ with $2\leq k \leq n$.\\
If the rank of the matrix $(h_{ij})$ is one, then
\[
H^2=|A|^2=C^2(H-\lambda)^2.
\]
From this equation, $H$ must be constant. Moreover, from
\[
|\nabla A|=|\nabla|A||=C|\nabla H| = 0,
\]
we conclude that $\Sigma$ is isoparametric and by Lawson's result $\Sigma$ must be $S^1_r(0)\times\mathbb{R}^{n-1}$.
\begin{flushright}
$\Box$
\end{flushright}

Theorem \ref{AM19} is a generalization of Ancari-Miranda theorem, for the case when $\lambda\geq 0$. Therefore, an immediate consequence of these theorems is as follows.
\begin{cor} Let $\Sigma^n\subset \mathbb{R}^{n+1}$ be a complete embedded CWMC hypersurface and $\delta \in \lbrace 0,1 \rbrace$. If $\Sigma$ satisfies the following properties:\\
(i) $H-\delta\lambda\geq 0$,\\
(ii) $\lambda\left( \text{tr}A^3(H-\lambda)+\frac{|A|^2}{2}\right)\leq 0$,\\
(iii) $\frac{1}{k^2}\int_{B^\Sigma_{2k}(p)\setminus B^\Sigma_k(p)} |A|^2e^{-f}\rightarrow 0$, when $k\rightarrow \infty$, for a fixed point $p\in \Sigma$,\\
then $\Sigma$ is either a hyperplane or a generalized cylinder $S^k_r(0)\times \mathbb{R}^{n-k}$, $1\leq k \leq n$.
\end{cor}
In \cite{cheng2015complete}, Cheng and Peng proved that a complete self-shrinker with $\inf H^2>0$ and constant norm of the second fundamental form must either be a sphere $S^n_{\sqrt{2n}}(0)$ or a cylinder $S^k_{\sqrt{2k}}(0)\times \mathbb{R}^{n-k}$, $1\leq k\leq n-1$. A consequence of the result above is the following result.
\begin{cor}Let $\Sigma\subset \mathbb{R}^{n+1}$ be a complete embedded self-shrinker with bounded norm of the second fundamental form. If $H\geq 0$ and
	\[
	\frac{1}{k^2}\int_{B^\Sigma_{2k}(p)\setminus B^\Sigma_k(p)} e^{-f}\rightarrow 0,
	\] 
	when $k\rightarrow \infty$, for a fixed point $p\in \Sigma$,
	then $\Sigma$ is either a hyperplane or a generalized cylinder $S^k_{\sqrt{2k}}(0)\times \mathbb{R}^{n-k}$, $1\leq k \leq n$.
\end{cor}
In the following, we prove Theorem \ref{teointro2} which is the third main result of this paper.\\

\noindent\textit{Proof of Theorem \ref{teointro2}}
First, we shall prove that the condition (i) implies that $\inf_{x\in\Sigma}H>-\infty$. In fact, if $\inf_{x\in\Sigma}H=-\infty$, then there exists a sequence $\lbrace p_j\rbrace_{j\in\mathbb{N}}$ in $\Sigma$ such that $H(p_j)\rightarrow -\infty$ when $j\rightarrow +\infty$. Therefore, for $j$ large enough $H(p_j)<0$ and $\frac{\lambda}{H(p_j)}<1$. From hypothesis (i), we have
\begin{eqnarray*}
|A|^2(p_j)\leq \frac{1}{2}\left(1-\frac{\lambda}{H(p_j)}\right)^{-1}.
\end{eqnarray*}
Since $\left(1-\frac{\lambda}{H(p_j)}\right)^{-1}\rightarrow 1$ when $j\rightarrow +\infty$, then the left side of the inequality above is bounded which implies that $|A|^2(p_j)$ is bounded. Therefore, $H^2(p_j)$ is also bounded, but this contradicts the assumption that $H(p_j)\rightarrow -\infty$.\\
Since $\inf_{x\in\Sigma}H>-\infty$, let us fix $C=\inf_{x\in\Sigma}H$ and from hypothesis (i), we have
\[
\mathcal{L}(H-C)=\frac{H}{2}+(\lambda-H)|A|^2\leq 0.
\]
By the maximum principle either $H-C>0$ or $H-C= 0$. If $H=C$, by Corollary \ref{cor1} $\Sigma$ is either a hyperplane or a generalized cylinder.\\
If $H-C>0$, computing $\Delta\log(H-C)$ we get
\begin{eqnarray*}
\Delta\log(H-C)&=&\text{div}(\nabla\log(H-C))\\
							&=&\text{div}\left(\frac{\nabla(H-C)}{H-C}\right)\\
							&=&\frac{1}{H-C}\Delta(H-C)-|\nabla\log(H-C)|^2.
\end{eqnarray*}
Therefore by hypothesis (ii), we obtain
\stepcounter{thm}
\begin{eqnarray*}
\mathcal{L}\log(H-C)&=&\frac{1}{H-C}\mathcal{L}(H-C)-|\nabla\log(H-C)|^2\\
							&=& \frac{\lambda-H}{H-C}|A|^2+\frac{H}{2(H-C)}-|\nabla\log(H-C)|^2\\
							&\leq & -|\nabla\log(H-C)|^2.
\end{eqnarray*}
Using integration by parts, for $\varphi\in C_0^\infty(\Sigma)$ we have
\begin{eqnarray*}
\int_\Sigma \varphi^2|\nabla\log(H-C)|^2e^{-f}&\leq &-\int_\Sigma \varphi^2\mathcal{L}\log(H-C)e^{-f}\\
													&=&\int_\Sigma \langle \nabla\log(H-C),\nabla \varphi^2\rangle e^{-f}\\
													&=&2\int_\Sigma\langle \varphi\nabla\log(H-C),\nabla \varphi\rangle e^{-f}\\
													&\leq & \int_\Sigma \left(\frac{1}{2}\varphi^2|\nabla\log(H-C)|^2+2|\nabla \varphi|^2\right) e^{-f}.
\end{eqnarray*}
Therefore
\[
\frac{1}{2}\int_\Sigma \varphi^2|\nabla\log(H-C)|^2e^{-f}\leq 2\int_\Sigma |\nabla \varphi|^2 e^{-f}.
\]
By hypothesis (iii), choosing a sequence as before and using the monotone convergence theorem, we conclude that
\[
|\nabla\log(H-C)|^2=0
\]
which implies that $H$ is constant, but this contradicts the assumption that $H>\inf_{x\in\Sigma} H$.
\begin{flushright}
$\Box$
\end{flushright}
An immediate consequence of Theorem \ref{teointro2} is the following.
\begin{cor}Let $\Sigma\subset \mathbb{R}^{n+1}$ be a complete embedded CWMC hypersurface with polynomial volume growth. If
\[
|A|^2(H-\lambda)\geq \frac{H}{2}
\]
then $\Sigma$ is either a hyperplane or a generalized cylinder $S^k_r(0)\times \mathbb{R}^{n-k}$, $1\leq k \leq n$.
\end{cor}
Another consequence of Theorem \ref{teointro2} is as follows.
\begin{cor}Let $\Sigma\subset \mathbb{R}^{n+1}$ be a complete embedded CWMC hypersurface. If the following conditions are satisfied:\\
(i) $H\geq 0$,\\
(ii) $\frac{H}{2}+\lambda|A|^2\leq 0$,\\
(iii) $\frac{1}{k^2}\int_{B^\Sigma_{2k}(p)\setminus B^\Sigma_k(p)} e^{-f}\rightarrow 0$, when $k\rightarrow \infty$, for a fixed point $p\in \Sigma$,\\
then $\Sigma$ is a hyperplane.
\end{cor}

For self-shrinkers, we obtain the following result.
\begin{cor}\label{corfinal}Let $\Sigma\subset \mathbb{R}^{n+1}$ be a complete embedded self-shrinker. If the following conditions are satisfied:\\
(i) $H\left(|A|^2- \frac{1}{2}\right)\geq 0$,\\
(ii)$\frac{1}{k^2}\int_{B^\Sigma_{2k}(p)\setminus B^\Sigma_k(p)} e^{-f}\rightarrow 0$, when $k\rightarrow \infty$, for a fixed point $p\in \Sigma$,\\
then $\Sigma$ is either a hyperplane or a generalized cylinder $S^k_{\sqrt{2k}}(0)\times \mathbb{R}^{n-k}$, $1\leq k \leq n$.
\end{cor}

\textbf{Acknowledgements:} We wish to express our gratitude to professor Xu Cheng for her support and useful suggestions. We thank professor Detang Zhou for his constant encouragement and motivation throughout this work. We also want to thank professor Thac Dung for his comment on Remark \ref{remthac}. 

\bibliographystyle{amsalpha}
\bibliography{Biblio}

\providecommand{\bysame}{\leavevmode\hbox to3em{\hrulefill}\thinspace}
\providecommand{\MR}{\relax\ifhmode\unskip\space\fi MR }
\providecommand{\MRhref}[2]{%
  \href{http://www.ams.org/mathscinet-getitem?mr=#1}{#2}
}
\providecommand{\href}[2]{#2}
\begin{thebibliography}{COW16}

\bibitem[AM20]{ancari2019volume}
Saul Ancari and Igor Miranda, \emph{Volume estimates and classification theorem
  for constant weighted mean curvature hypersurfaces}, The Journal of Geometric
  Analysis (2020), https://doi.org/10.1007/s12220--020--00413--2.

\bibitem[CL13]{cao2013gap}
Huai-Dong Cao and Haizhong Li, \emph{A gap theorem for self-shrinkers of the
  mean curvature flow in arbitrary codimension}, Calculus of Variations and
  Partial Differential Equations \textbf{46} (2013), no.~3-4, 879--889.

\bibitem[CM12]{colding2012generic}
Tobias~H Colding and William~P Minicozzi, \emph{Generic mean curvature flow i;
  generic singularities}, Annals of Mathematics (2012), 755--833.

\bibitem[COW16]{cheng2016rigidity}
Qing-Ming Cheng, Shiho Ogata, and Guoxin Wei, \emph{Rigidity theorems of
  $\lambda$-hypersurfaces}, Communications in Analysis and Geometry \textbf{24}
  (2016), no.~1, 45--58.

\bibitem[CP15]{cheng2015complete}
Qing-Ming Cheng and Yejuan Peng, \emph{Complete self-shrinkers of the mean
  curvature flow}, Calculus of Variations and Partial Differential Equations
  \textbf{52} (2015), no.~3-4, 497--506.

\bibitem[CW18]{cheng2018complete}
Qing-Ming Cheng and Guoxin Wei, \emph{Complete $\lambda$-hypersurfaces of
  weighted volume-preserving mean curvature flow}, Calculus of Variations and
  Partial Differential Equations \textbf{57} (2018), no.~2, 32.

\bibitem[Gua18]{guang2018gap}
Qiang Guang, \emph{Gap and rigidity theorems of $\lambda$-hypersurfaces},
  Proceedings of the American Mathematical Society \textbf{146} (2018), no.~10,
  4459--4471.

\bibitem[Hal12]{halldorsson2012self}
Hoeskuldur~P Halldorsson, \emph{Self-similar solutions to the curve shortening
  flow}, Transactions of the American Mathematical Society (2012), 5285--5309.

\bibitem[Law69]{lawson1969local}
H~Blaine Lawson, \emph{Local rigidity theorems for minimal hypersurfaces},
  Annals of Mathematics (1969), 187--197.

\bibitem[LS11]{le2011blow}
Nam~Q Le and Natasa Sesum, \emph{Blow-up rate of the mean curvature during the
  mean curvature flow and a gap theorem for self-shrinkers}, Communications in
  Analysis and Geometry \textbf{19} (2011), no.~4, 633--659.

\bibitem[MR15]{mcgonagle2015hyperplane}
Matthew McGonagle and John Ross, \emph{The hyperplane is the only stable,
  smooth solution to the isoperimetric problem in gaussian space}, Geometriae
  Dedicata \textbf{178} (2015), no.~1, 277--296.

\bibitem[Sun18]{sun2018compactness}
Ao~Sun, \emph{Compactness and rigidity of $\lambda $-surfaces}, arXiv preprint
  arXiv:1804.09316 (2018).

\bibitem[WP19]{wei2019note}
Guoxin Wei and Yejuan Peng, \emph{A note on rigidity theorem of
  $\lambda$-hypersurfaces}, Proceedings of the Royal Society of Edinburgh
  Section A: Mathematics \textbf{149} (2019), no.~6, 1595--1601.

\end{thebibliography}
\end{document}